\documentclass[runningheads]{llncs}
\usepackage{graphicx}
\usepackage[T2A]{fontenc}
\usepackage[utf8]{inputenc}
\usepackage[russian,english]{babel}
\usepackage{amsmath}
\usepackage{amsfonts}
\usepackage{wrapfig}
\usepackage{float}
\usepackage{bbm}
\usepackage{xcolor}
\usepackage{nicefrac}
\usepackage{mathtools}

\makeatletter 
\newcommand\mynobreakpar{\par\nobreak\@afterheading} 
\makeatother

\DeclarePairedDelimiter{\ceil}{\lceil}{\rceil}

\usepackage[ruled,vlined]{algorithm2e}

\usepackage{amsmath}
\usepackage{amssymb}

\newsavebox{\mybox}
\newlength{\mywidth}
\newlength{\myheight}
\newlength{\myline}
\newlength{\myoffset}
\newcommand{\mysqrt}[1]%
{\setlength{\myline}{.2ex}%
\addtolength{\myline}{.06pt}%
\setlength{\myoffset}{.9em}
\addtolength{\myoffset}{-2pt}
\savebox{\mybox}{$\displaystyle\sqrt{#1}$}%
\settoheight{\myheight}{\usebox{\mybox}}%
\addtolength{\myheight}{-.3ex}
\settowidth{\mywidth}{\usebox{\mybox}}%
\addtolength{\mywidth}{-\myoffset}%
 \rlap{\usebox{\mybox}}\hspace{\myoffset}{\raisebox{\myheight}{\rule{\mywidth}{\myline}}}}

\begin{document}
\title{On the Computational Efficiency of Catalyst Accelerated Coordinate Descent\thanks{D.A. Pasechnyuk's research was supported by the A.M. Raigorodsky Scholarship in the field of optimization and RFBR grant 19-31-51001 (Scientific mentoring). The work of V.V. Matyukhin was supported by the Ministry of Science and Higher Education of the Russian Federation (state assignment) No. 075-00337-20-03, project number 0714-2020-0005.}}

\titlerunning{On the Computational Efficiency of Catalyst CDM}
%
\author{Dmitry Pasechnyuk\inst{1}\orcidID{0000-0002-1208-1659}\and
Vladislav Matyukhin\inst{1}\orcidID{0000-0003-4961-1577}}
\authorrunning{D. Pasechnyuk, V. Matyukhin}
%
\institute{
Moscow Institute of Physics and Technology (National Research University), Dolgoprudny, Russia\\
\email{\{pasechniuk.da,matyukhin\}@phystech.edu}}
\maketitle              
\begin{abstract}
This article is devoted to one particular case of using universal accelerated proximal envelopes to obtain computationally efficient accelerated versions of methods used to solve various optimization problem setups. We propose a proximally accelerated coordinate descent method that achieves the efficient algorithmic complexity of iteration and allows taking advantage of the data sparseness. It was considered an example of applying the proposed approach to optimizing a SoftMax-like function, for which the described method allowing weaken the dependence of the computational complexity on the dimension $n$ in $\mathcal{O}(\sqrt{n})$ times and, in practice, demonstrates a faster convergence in comparison with standard methods. As an example of applying the proposed approach, it was shown a variant of obtaining on its basis some efficient methods for optimizing Markov Decision Processes (MDP) in a minimax formulation with a Nesterov smoothed target functional. 

\end{abstract}

\section{Introduction}

One of the most important theoretical results in convex optimization was the development of accelerated optimization methods \cite{nesterov2018lectures}. At the initial stage of implementation of this concept, many accelerated algorithms for different problem setups were proposed. But each such case required special consideration of the possibility of acceleration. Therefore, the proposed designs were significantly different and did not allowing assume a way to generalize them. A significant step towards the development of a universal scheme for accelerating optimization methods was the work in which an algorithm called Catalyst proposed, based on the idea of \cite{parikh2014proximal,rockafellar1976monotone} and allowing to accelerate other optimization methods, using them for the sequential solution of several Moreau-Yosida regularized auxiliary problems \cite{lin2015universal,lin2017catalyst}. Following these ideas, many variants of the applications of this method and its modifications \cite{ivanova2019adaptive,kulunchakov2019generic,paquette2017catalyst} were proposed. Among the most recent results until the time of writing this paper, the generalizations of the discussed approach to tensor methods \cite{bubeck2019near,monteiro2013accelerated,doikov2020contracting,dvinskikh2020accelerated} were also described. The corresponding representation of the accelerated proximal envelope to the authors' knowledge is the most general of those described in the literature, and therefore, this work is focused primarily on the methods proposed in the works \cite{doikov2020contracting,dvinskikh2020accelerated}.

The main motivation of this work is to describe the possibilities of the practical application of universal accelerated proximal envelopes for constructing computationally and oracle efficient optimization methods. Let us consider the classical coordinate descent method \cite{bubeck2014convex}, the iteration of which for the convex function $f: \mathbb{R}^n \rightarrow \mathbb{R}$ has the form:
$$x_{k+1}^i = x_k^i - \eta \nabla_i f(x_k), \quad i \sim \mathcal{U}\{1,...,n\},\;\;\eta > 0.$$
One of the many applications of this method is the optimization of functionals, where the calculation of the one component of the gradient is significantly more efficient than the calculation of the full gradient vector of these functionals (in particular, many problems in the case of sparse formulations satisfy this condition). However, the oracle complexity of this method, provided that the method stops when the $\varepsilon$-small residual by the function value is reached, is $\mathcal{O}\left(n \frac{\overline{L} R^2}{\varepsilon}\right)$, where $R^2 = \|x_0 - x_*\|_2^2$, $\overline{L}=\frac{1}{n} \sum_{i=1}^n L_i$ is the average of the Lipschitz constants of the gradient components; moreover, this estimate is not optimal for the class of convex problems. Let us now consider the accelerated coordinate descent method proposed by Yu.E.~Nesterov \cite{nesterov2017efficiency}, the oracle complexity of this method corresponds to the optimal estimate: $\mathcal{O}\left(n\sqrt{\frac{\widetilde{L} R^2}{\varepsilon}}\right)$, where $\sqrt{\widetilde{L}} = \frac{1}{n} \sum_{i=1}^n \sqrt{L_i}$ is the mean of square roots of the Lipschitz constants of the gradient components. At the same time, the situation changes drastically when the algorithmic complexity of the method considered: even if the computation of one component of the gradient has the complexity $\mathcal{O}(s)$, $s \ll n$, the complexity of iterating the accelerated coordinate descent method will be $\mathcal{O}(n)$, unlike the standard method, the iteration complexity of which is $\mathcal{O}(s)$, this means, essentially, that the degree of the sparseness of the problem when using the accelerated coordinate descent method does not significantly affect the complexity of the algorithm, and besides, the complexity in this case quadratically depends on the dimension of the problem: together, this to some extent devalues the use of the coordinate descent method in this case. Thus, an interesting problem is the construction of an accelerated coordinate descent method, the iteration complexity of which, as in the standard version of the method, is $\mathcal{O}(s)$, this is possible due to the application of the universal accelerated proximal envelope ``Accelerated Meta-algorithm'' \cite{dvinskikh2020accelerated}. Note that the approach to accelerating the cordinate descent method considered in this work is, in fact, one of the cases of applying the technique described in \cite{ivanova2019adaptive}. Thus, the content of the paper is devoted to a more detailed analysis of this particular approach, including implementation features and possible applications.

This article consists of an introduction, conclusion and the main Section~\ref{section1}. It describes the theoretical results on the convergence and algorithmic complexity of the coordinate descent method, accelerated by using the ``Accelerated Meta-algorithm'' envelope (Section \ref{th}). Using the example of the SoftMax-like function optimization problem, it was experimentally tested method's effectiveness with relation to its working time, there were described the possibilities of its computationally efficient implementation, also it was carried out a comparison with standard methods (Section \ref{exp}). Further, as an example of applying the proposed approach, it was provided a method for optimizing Markov Decision Processes in a minimax formulation, based on applying the method introduced in this paper to the Nesterov smoothed target functional. The proposed approach obtains estimates close to that for several efficient and practical methods for optimizing the discounted MDP and matches the best estimates for the averaged MDP problem (Section \ref{mdp}).

\section{Accelerated Meta-algorithm and coordinate descent} \label{section1}
\subsection{Theoretical guarantees} \label{th}

Let us consider the following optimization problem of the function $f: \mathbb{R}^n \rightarrow \mathbb{R}$:
$$\min_{x \in \mathbb{R}^n} f(x),$$
subject to:
\begin{enumerate}
    \item $f$ is differentiable on $\mathbb{R}^n$;
    \item $f$ is convex on $\mathbb{R}^n$;
    \item $\nabla_i f$ is component-wise Lipschitz continuous, i.e. $\forall x \in \mathbb{R}^n$ and $u \in \mathbb{R},$ $\exists~L_{i}~\in~\mathbb{R} \; (i = 1, \ldots, n)$, such that 
    $$\left|\nabla_{i} f\left(x+u e_{i}\right)-\nabla_{i} f(x)\right| \leq  L_i|u|,$$
    where $e_i$ is the $i$-th unit basis vector, $i \in \{1,...,n\}$;
    \item $\nabla f$ is $L$-Lipschitz contiouous.
\end{enumerate}

Let us turn to the content of the work \cite{dvinskikh2020accelerated}, where a general version of the ``Accelerated Meta-algorithm'' for solving convex optimization problems for composite functionals of the form $F(x) = f(x) + g(x)$ was proposed. For the considered formulation of the problem, such generality not required; it is sufficient to apply a special case of the described scheme for $p = 1$, $f \equiv 0$ (using the designations of the corresponding work), in which the described envelope takes the form of an accelerated proximal method. This method is listed as Algorithm~\ref{am}.

\begin{algorithm}[t] \label{am}
\SetAlgoLined
    \textbf{Input:} $H > 0$, $x_0 \in \mathbb{R}^n$\;
    \vspace{0.2cm}

    $\lambda \leftarrow \nicefrac{1}{2H}$\;
    $A_0 \leftarrow 0$; $v_0 \leftarrow x_0$\;
    \vspace{0.2cm}
    
    \For{k = 0, ..., $\widetilde{N}-1$}
    {
        $\displaystyle a_{k+1} \leftarrow \frac{\lambda + \sqrt{\lambda^2 + 4 \lambda A_k}}{2}$\;
        $A_{k+1} \leftarrow A_k + a_{k+1}$\;
        
        \vspace{0.2cm}
        $\displaystyle \widetilde{x}_k \leftarrow \frac{A_k v_k + a_{k+1} x_k}{A_{k+1}}$\;
        
        \vspace{0.2cm}
        By running the method $\mathcal{M}$,\\
        find the solution of the following auxiliary problem\\
        with an accuracy $\varepsilon$ by the argument:\\
        \vspace{0.1cm}
        $\displaystyle v_{k+1} \in \text{Arg}^{\varepsilon} \min_{y \in \mathbb{R}^n} \left\{ f(y) + \frac{H}{2}\|y - \widetilde{x}_k\|^2_2 \right\}$\;
        \vspace{0.5cm}
        
        $\displaystyle x_{k+1} \leftarrow x_k - a_{k+1} \nabla f(v_{k+1})$\;
    }
    
    \Return $v_{\widetilde{N}}$\;
    
    \caption{Accelerated Meta-algorithm for First-order Method $\mathcal{M}$}
\end{algorithm}

Before formulating any results on the convergence of the proposed accelerated coordinate descent method described below, it is necessary to start with a detailed consideration of the process of solving the auxiliary problems, where its analytical solution is available only in rare cases. Therefore, one should apply numerical methods to find its approximate solution, and that is inaccurate. The solving process of the auxiliary problem will be until the following stop condition is satisfied (\cite{kamzolov2020optimal}, Appendix B):
\begin{equation} \label{monteiro}
    \left\|\nabla\left\{F(y_\star) := f(y_\star) + \frac{H}{2}\|y_\star-\widetilde{x}_k\|^2_2\right\}\right\|_2 \leq \frac{H}{2}\|y_\star-\widetilde{x}_k\|_2,
\end{equation}
where $y_\star$ is an approximate solution of the auxiliary problem, returned by internal method $\mathcal{M}$. Due to the $\|\nabla F(y_*) \|_2 = 0$ (where $y_*$ denotes an exact solution of the considered problem), and due to the $(L + H)$-Lipschitz continuity of $\nabla F$, we have got:
\begin{equation} \label{smooth}
    \|\nabla F(y_\star)\|_2 \leq (L + H) \|y_\star - y_*\|_2.
\end{equation}
Writing out the triangle inequality:
$\|\widetilde{x}_k - y_*\|_2 - \|y_\star - y_*\|_2 \leq \|y_\star - \widetilde{x}_k\|_2$,
and using together the inequalities \eqref{monteiro}, \eqref{smooth}, we have got the final view of stop condition:
\begin{equation} \label{crit}
    \|y_\star - y_*\|_2 \leq \frac{H}{3H + 2L} \|\widetilde{x}_k - y_*\|_2.
\end{equation}

\begin{algorithm}[t] 
\SetAlgoLined
    \textbf{Input:} $y_0 \in \mathbb{R}^n$\;
    \vspace{0.2cm}
    
    $Z \leftarrow \sum_{i=1}^n (H + L_i)$\;
    $p_i \leftarrow (H + L_i) / Z,\quad i \in \{1,...,n\}$\;
    Discrete probability distribution $\pi$\\
    with probabilities $p_i$\;
    \vspace{0.2cm}
    
    \For{k = 0, ..., $N-1$}
    {
        $i \sim \pi\{1,...,n\}$\;
        $y_{k+1} \leftarrow y_k$\;
        $\displaystyle y_{k+1}^i = y_k^i - \frac{1}{H + L_i} \nabla_i F(y_k)$\; 
    }
    
    \Return $y_N$\;
    
    \caption{Coordinate descent method}
    \label{cdm}
\end{algorithm}

Essentially this implies that the required argument accuracy of solving the auxiliary problem does not depend on the accuracy of the main problem. That makes it possible to simplify the obtaining of further results.

Let us now consider the main method used for solving auxiliary problems: the content of the coordinate descent method \cite{nesterov2012efficiency} (in the particular case, when $\gamma = 1$) is listed as Algorihtm~\ref{cdm}. For this method, in the case of the considered auxiliary problems, the following result holds:

\begin{theorem}{(\cite{bubeck2014convex}, theorem 6.8)} \label{cdm_conv}
    Let $F$ be $H$-strongly convex function with respect to $\|\cdot\|_2$. Then for the sequence $ \{y_k \}_{k = 1}^N $ generated by the described coordinate descent algorithm~\ref{cdm}, it holds the following inequality:
    \begin{align}\label{f_conv}
        &\mathbb{E}[F(y_N)] - F(y_*) \leq \left(1 - \frac{1}{\kappa}\right)^N (F(y_0) - F(y_*)),\\
        &\text{where}\quad\kappa = \frac{H}{Z},\;\;Z = \sum_{i=1}^n (H + L_i),
    \end{align}
    where $\mathbb{E}[\cdot]$ denotes the mathematical expectation of the specified random variable with respect to the randomness of methods trajectory induced by a random choice of components $i$ at each iteration.
\end{theorem}

\noindent Using this result, formulate the following statement on the number of iterations of the coordinate descent sufficient to satisfy the stop condition \eqref{crit}.

\begin{corollary} \label{cor_num}
    The expectation $\mathbb{E} [y_N] $ of the point resulting from the coordinate descent method (Algorithm~\ref{cdm}) satisfies the condition \eqref{crit} if the following inequality on iterations number holds:
    \begin{align} \label{cor}
        &N \geq N(\widetilde{\varepsilon}) = \ceil[\Bigg]{\frac{Z}{H} \ln{\left\{\left(1 + \frac{L}{H}\right) \left(3 + \frac{2L}{H}\right)^2\right\}} },\\ &\text{where}\quad\widetilde{\varepsilon} = \frac{H}{2} \left(\frac{H}{3H + 2L}\right)^2 \|y_0 - y_*\|_2^2.
    \end{align}
\end{corollary}
\begin{proof}
    is in Appendix \ref{cor_num_proof}.
\end{proof}

Now that the question of the required accuracy and oracle complexity of solving the auxiliary problem using the proposed coordinate descent method is clarified, we can proceed to the results on the convergence of the Accelerated Meta-algorithm. For the used stop condition \eqref{crit} of the method that solves the auxiliary problem, the following result on the convergence of the Accelerated Meta-algorithm holds.

\begin{theorem}{(\cite{dvinskikh2020accelerated}, theorem 1)} \label{am_conv}
    For $H > 0$ and the sequence $\{v_k\}_{k=1}^{\widetilde{N}}$ generated by the Accelerated Meta-algorithm with some non-stochastic internal method, it holds the following inequality:
    \begin{equation} \label{am_th}
        f(v_{\widetilde{N}}) - f(x_*) \leq \frac{48}{5} \frac{H \|x_0 - x_*\|_2^2}{\widetilde{N}^2}.
    \end{equation}
\end{theorem}

\noindent Based on the last statement, one can formulate a theorem on the convergence of the Accelerated Meta-algorithm in the case of using the stochastic method and, in particular, coordinate gradient descent method.

\begin{theorem} \label{am_stoch_conv}
    For $H > 0$ and some $0 < \delta < 1$, the point $v_{\widetilde{N}}$ resulting from the Accelerated Meta-algorithm uses coordinate descent method to solve the auxiliary problem, solving it $N_{\delta}$ iterations, satisfies the condition
    \begin{equation*}
        Pr(f(v_{\widetilde{N}}) - f(x_*) < \varepsilon) \geq 1 - \delta,
    \end{equation*}
    where $Pr(\cdot)$ denotes the probability of the specified event, if
    \begin{align} \label{iters_out_inn}
        &\widetilde{N} \geq \ceil[\Bigg]{\frac{4 \sqrt{15}}{5} \sqrt{\frac{H \|x_0 - x_*\|_2^2}{\varepsilon}}\;}, \\
        &N_{\delta} \geq N\left(\frac{\widetilde{\varepsilon}\delta}{\widetilde{N}}\right) = \ceil[\Bigg]{\frac{Z}{H} \ln { \left\{\frac{\widetilde{N}}{\delta} \left(1 + \frac{L}{H}\right) \left(3 + \frac{2L}{H}\right)^2\right\}} }.
    \end{align}
\end{theorem}

\begin{proof}
    The corollary \ref{cor_num} presents an estimate of the number of iterations sufficient to satisfy the following condition for the expected value of the function at the resulting point of the method:
    $$\mathbb{E}[F(y_{N(\widetilde{\varepsilon})})] - F(y_*) \leq \widetilde{\varepsilon}.$$
    Use the Markov inequality and obtain the formulation of this condition in terms of the estimate of the probability of large deviations \cite{anikin2015modern}: deliberately choose the admissible value of the probability of non-fulfillment of the stated condition, so that
    $0 < \delta/\widetilde{N} < 1$, where $\widetilde{N}$ expressed from \eqref{am_th}; then
    $$Pr\left(F\left(y_{N(\widetilde{\varepsilon} \delta/\widetilde{N})}\right) - F(y_*) \geq \widetilde{\varepsilon}\right) \leq \frac{\delta}{\widetilde{N}} \cdot \frac{\mathbb{E}\left[F\left(y_{N(\widetilde{\varepsilon} \delta/\widetilde{N})}\right)\right] - F(y_*)}{\widetilde{\varepsilon} \cdot \delta/\widetilde{N}} = \frac{\delta}{\widetilde{N}}.$$
    Since the probability that the obtained solution of some separately taken auxiliary problem will not satisfy the stated condition is equal to $\delta / \widetilde{N}$, it means that the probability that for $\widetilde{N}$ iterations of the Accelerated Meta-algorithm the condition will not be satisfied although if for one of the problems, there is $\widetilde{N} \cdot \delta / \widetilde{N} = \delta$, whence the proved statement follows.
\end{proof}

Further, combining the estimates given in Theorem~\ref{am_stoch_conv}, we can obtain an asymptotic estimate for the total number of iterations of the coordinate descent method, sufficient to obtain a solution of the considered optimization problem with a certain specified accuracy, as well as an estimate for the optimal $H$ value:

\begin{corollary} \label{cor_est}
    In order to the point $v_{\widetilde{N}}$, which is the result of the Accelerated Meta-algorithm, to satisfy the condition
    $$Pr(f(v_{\widetilde{N}}) - f(x_*) < \varepsilon) \geq 1 - \delta,$$
    it is sufficient to perform a total of
    \begin{equation} \label{total_iters}
        \hat{N} \geq \widetilde{N} \cdot N_\delta = \mathcal{O}\left(\frac{Z \|x_0 - x_*\|_2}{\sqrt{H}} \cdot \frac{1}{\varepsilon^{1/2}} \log\left\{\frac{1}{\varepsilon^{1/2} \delta}\right\}\right)
    \end{equation}
    iterations of coordinate descent method to solve the auxiliary problem. In this case, the optimal value of the regularization parameter $H$ of the auxiliary problem should be chosen as $H \simeq \frac{1}{n} \sum_{i = 1}^n L_i$ ($\simeq$ denotes equality up to a small factor of the $\log$ order).
\end{corollary}

\begin{proof}
    The expression for $\hat{N}$ can be obtained by the direct substitution of one of the estimates given in \eqref{iters_out_inn} into another, and their subsequent multiplication. If we exclude from consideration a small factor of order $\log(L / H)$, the constant in the estimate will depend on $H$ as:
    $$\sqrt{H} \cdot \frac{Z/n}{H} = \sqrt{H} \left(1 + \frac{\frac{1}{n} \sum_{i=1}^n L_i}{H}\right).$$
    By minimizing the presented expression by $H$, we get the specified result.
\end{proof}

A similar statement can be formulated for the expectation of the number of total iterations of the coordinate descent method, without resorting to estimates of the probabilities of large deviations, following the reasoning scheme proposed in \cite{daspremont2021acceleration}:

\begin{theorem} \label{th4}
    The expectation $\mathbb{E} [\hat{N}]$ of the sufficient total number of the iterations of the coordinate descent method, to obtain a point $v_{\widetilde{N}}$ satisfying the following condition:
    $$f(v_{\widetilde{N}}) - f(x_*) < \varepsilon$$
    can be bounded as follows:
    $$\mathbb{E} [\hat{N}] \leq \widetilde{N} \cdot (N(\widetilde{\varepsilon}) + 1) = \mathcal{O}\left(\sqrt{\frac{\overline{L} \|x_0 - x_*\|^2_2}{\varepsilon}}\right),\quad\text{ where}\quad \overline{L} = Z / n = \frac{1}{n} \sum_{i=1}^n L_i.$$
\end{theorem}

\begin{proof}
    is in Appendix \ref{th4_proof}.
\end{proof}

As we can see, the proposed scheme of reasoning allows us to reduce the logarithmic factor in estimating the number of iterations of the method. However, this result is less constructive than the one presented in Corollary \ref{cor_est}. Indeed, in the above reasoning, we operated with the number of iterations $N$, after which the stopping condition for the internal method is satisfied, but in the program implementation, stopping immediately after fulfilment of this condition is not possible, if only because, which is impossible the verification of this criterion due to the natural lack of information about $y_*$. So the last result is more relevant from the point of view of evaluating the theoretical effectiveness of the method, while when considering specific practical cases, one should apply the estimate~\eqref{total_iters}. 

Let us now consider in more detail the issue of the algorithmic complexity of the proposed accelerated coordinate gradient descent method. 

\begin{theorem} \label{th5}
    Let the complexity of computing one component of the gradient of $f$ is $\mathcal{O}(s)$. Then the algorithmic complexity of the Accelerated Meta-algorithm with coordinate descent as internal method is
    $$T = \mathcal{O}\left(s \cdot n \cdot \sqrt{\frac{\overline{L} \|x_0 - x_*\|_2^2}{\varepsilon}} \log\left\{\frac{1}{\varepsilon^{1/2} \delta}\right\}\right).$$
\end{theorem}

\begin{proof}
    is in Appendix \ref{th5_proof}.
\end{proof}
Note also that the memory complexity of the method is $\mathcal{O}(n)$, as well as the complexity of the preliminary calculations (for the coordinate descent method, there is no need to perform them again for every iteration).

Let us compare the estimates obtained for proposed approach (Catalyst CDM) with estimates of other methods that can be used to solve problems in the described setting: Fast Gradient Method (FGM), classical Coordinate Descent Method (CDM) and Accelerated Coordinate Descent Method in the version of Yu.E.~Nesterov (ACDM). The estimates are shown in Table~\ref{table_compare}, below. As can be seen from the above asymptotic estimates of the computational complexity, the proposed method allows to achieve a convergence rate that is not inferior to other methods with respect to the nature of the dependence on the dimension of the problem $n$ and the required accuracy $\varepsilon$, at a certain price for this in the form of a logarithmic factor assessment. Note, in addition, that despite the significant similarity of estimates, between the two most efficient methods in the table (FGM and Catalyst CDM) there is also a difference in the constants characterizing the smoothness of the function, i.e. $L$ in FGM and $\overline{L}$ in Catalyst CDM, thus the behavior of the considered method for various problems directly depends on the character of its componentwise smoothness.

\begin{table}[ht]
\centering
\caption{Comparison of the effectiveness of methods}
\label{table_compare}
\begin{tabular}{|c|c|c|c|}
\hline
Algorithm        & Iteration complexity                 & Comp. complexity       &  Source \\ \hline
FGM          & $ \mathcal{O}\left(s \cdot n\right)$ & $\displaystyle \mathcal{O}\left(s \cdot n \cdot \frac{1}{\varepsilon^{1/2}}\right)$ & \cite{nesterov2018lectures} \\ \hline
CDM          & $ \mathcal{O}\left(s\right)$  & $\displaystyle \mathcal{O}\left(s \cdot n \cdot \frac{1}{\varepsilon}\right)$ & \cite{bubeck2014convex} \\ \hline
ACDM         & $ \mathcal{O}\left(n\right)$  & $\displaystyle \mathcal{O}\left(n^2 \cdot \frac{1}{\varepsilon^{1/2}}\right)$ & \cite{nesterov2017efficiency} \\ \hline
Catalyst CDM & $ \mathcal{O}\left(s\right)$  & $\displaystyle \widetilde{\mathcal{O}}\left(s \cdot n \cdot \frac{1}{\varepsilon^{1/2}}\right)$ & this paper \\ \hline
\end{tabular}
\end{table}

\subsection{Numerical experiments} \label{exp}

This section describes the character of the practical behavior of the proposed method by the example of the following optimization problem for the SoftMax-like function:
\begin{equation}\label{eq:softmax}
\min\limits_{x \in \mathbb{R}^n}~~ \{f(x) = \gamma \ln \left( \sum_{j=1}^m \exp\left(\frac{\left[A x\right]_j}{\gamma}\right) \right) -\langle b,x \rangle\},
\end{equation}
where $b \in \mathbb{R}^n$, $A \in \mathbb{R}^{m \times n}, \gamma \in \mathbb{R}_{+}$. Problems of this kind are essential for many applications, in particular, they arise in entropy-linear programming problems as a dual problem \cite{chernov2016method,gasnikov2016effective}, in particular in optimal transport problem, is also a smoothed approximation of the $\max$ function  (which gave the function the name SoftMax) and, accordingly, the norm $\|\cdot\|_{\infty}$, which may be needed in some formulations of the PageRank problem or for solving systems of linear equations. Moreover, in all the described problems, an important special case is the sparse setting, in which the matrix $A$ is sparse, that is, the average number of nonzero entries in the row $A_j$ does not exceed some $s \ll n$ (it will also be convenient to assume the possibility one of the strings $A_j$ is completely non-sparse).

Let us formulate the possessed  properties by the function $f$\;\cite{gasnikov2018modern}:
\begin{enumerate}
    \item $f$ is differentiable;\mynobreakpar
    \item $\nabla f$ satisfies the Lipschitz condition with the constant $L = \max_{j=1,...,m} \|A_j\|_2^2$;\mynobreakpar
    \item $\nabla_i f$ satisfy the component-wise Lipschitz condition with the constants $L_i = \max_{j=1,...,m} |A_{j i}|$.
\end{enumerate}

Let us write the expression for the $i$-th component of the gradient of the function $f$:
$$\nabla_i f(x) = \frac{\sum_{j=1}^m{A_{j i} \exp\left(\left[A x\right]_j\right)}}{\sum_{j=1}^m{\exp\left(\left[A x\right]_j\right)}}.$$
As we can see, the naive calculation of this expression can take time comparable to the calculation of the whole gradient and it will significantly affect the computational complexity, and hence the working time of the method. However, at the same time, many terms in this expression can be recalculated either infrequently or in a component-wise manner, and used as members of methods additional sequences when performing a step of method, so that the complexity of the iteration will remain efficient, and the use of the coordinate descent methods will be justified. For the convenience of describing the computational used methods, we write the step of the coordinate descent algorithm in such form:
$$y_{k+1} = y_{k} + \eta e_i,$$
where $\eta$ is the step size, multiplied by the corresponding gradient component, $e_i$ is the $i$-th unit basis vector. 

In order to the performed experiments:
\begin{enumerate}
    \item We will store a sequence of values $\left\{\exp\left(\left[A y_k\right]_j\right)\right\}_{j=1}^m$, used to calculate the sum in the numerator. Updating these values after executing a method step takes algorithmic complexity $\mathcal{O}(s)$, due to the $A y_{k+1} = A y_k + \eta A_i$, and at the same time $A_i$ has at most $s$ nonzero components, which means that it will be necessary to calculate only less than or equal to $s$ correcting factors and multiply the corresponding values from the sequence by them.
    \item From the first point, it can be understood that the multiplication of sparse vectors should be performed in $\mathcal{O}(s)$, considering only nonzero components. In terms of program implementation, this means the need to use a sparse representation for cached values and for rows of the matrix $A$, that is, storing only index-value pairs for all nonzero elements. Then, obviously, the complexity of arithmetic operations for such vectors will be proportional to the complexity of a loop with elementary arithmetic operations, the number of iterations of which is equal to the number of nonzero elements (in the python programming language, for example, this storage format is implemented in the method scipy.sparse.csr\_matrix \cite{scipy}). 
    \item Similarly, we will store the value $\sum_{j=1}^m{\exp\left(\left[A y_k\right]_j\right)}$, which is the denominator of the presented expression. Its updating is carried out with the same complexity as updating a sequence (by calculating the sum of nonzero terms added to each value from the sequence).
    \item Since evaluating the specified expression requires evaluating exponent values, type overflow can occur. To solve this problem, the standard technique is exp-normalize trick \cite{blanchard2019accurately}. However, to use it, one should also store the value $\max_{j=1,...,m} \left[A y_k\right]_j$. At the same time, there is no need to maintain exactly this value, or, otherwise, its approximation to keep the exponents values small, so this value can be recalculated much rarely: for example, once in $m$ iterations (in this case, the amortized complexity will also be equal to $\mathcal{O}(s)$).
\end{enumerate}

So, in the further reasoning, one can assume that the iteration of the coordinate descent algorithm for solving the corresponding auxiliary problem has amortized complexity $\mathcal{O}(s)$.

Further, let us consider in more detail the question of the values of the smoothness constants of this functional. We can write down asymptotic formulas for
$L$ and $\overline{L} = \frac{1}{n} \sum_{i=1}^n L_i$:
$$L = \max_{j=1,...,m} \|A_j\|_2^2 = \mathcal{O}(n), \quad \overline{L} = \frac{1}{n} \sum_{i=1}^n \max_{j=1,...,m} |A_{j i}| = \mathcal{O}(1).$$
Using these estimates, let us refine the computational complexity of the FGM and Catalyst CDM (CCDM) methods as applied to this problem:
$$T_{FGM} = \mathcal{O}\left(s \cdot n^{3/2} \cdot \frac{1}{\varepsilon^{1/2}}\right), \quad T_{CCDM} = \widetilde{\mathcal{O}}\left(s \cdot n \cdot \frac{1}{\varepsilon^{1/2}}\right).$$
Thus, in theory, the application of the Catalyst CDM method for solving this problem allows, in comparison with FGM, to reduce the factor of order $\mathcal{O}(\sqrt{n})$ in the asymptotic estimate of the computational complexity. In practice, this means that it is reasonable to apply the proposed method to problems of large dimensions.

Let us now compare the performance of the proposed approach (Catalyst CDM) with a number of alternative approaches: Gradient Method (GM), Fast Gradient Method (FGM), Coordinate Descent Method (CDM) and Accelerated Coordinate Descent Method (ACDM), by the example of the problem~\eqref{eq:softmax} with an artificially generated matrix $A$ in two different ways. Fig. \ref{fig:umc1} and \ref{fig:umc2} present plots of the convergence of the methods under consideration: in x-axis we show the running time of the methods in seconds, and in the y-axis we show the residual of the function on a logarithmic scale ($f_*$ calculated by searching for the corresponding point $x_*$ using the FGM method, tuned for an accuracy that is obviously much higher than that possible to achieve at the selected time interval).

\begin{figure}[!ht]
\begin{minipage}{0.49\textwidth}
    \centering
    \includegraphics[width=\linewidth]{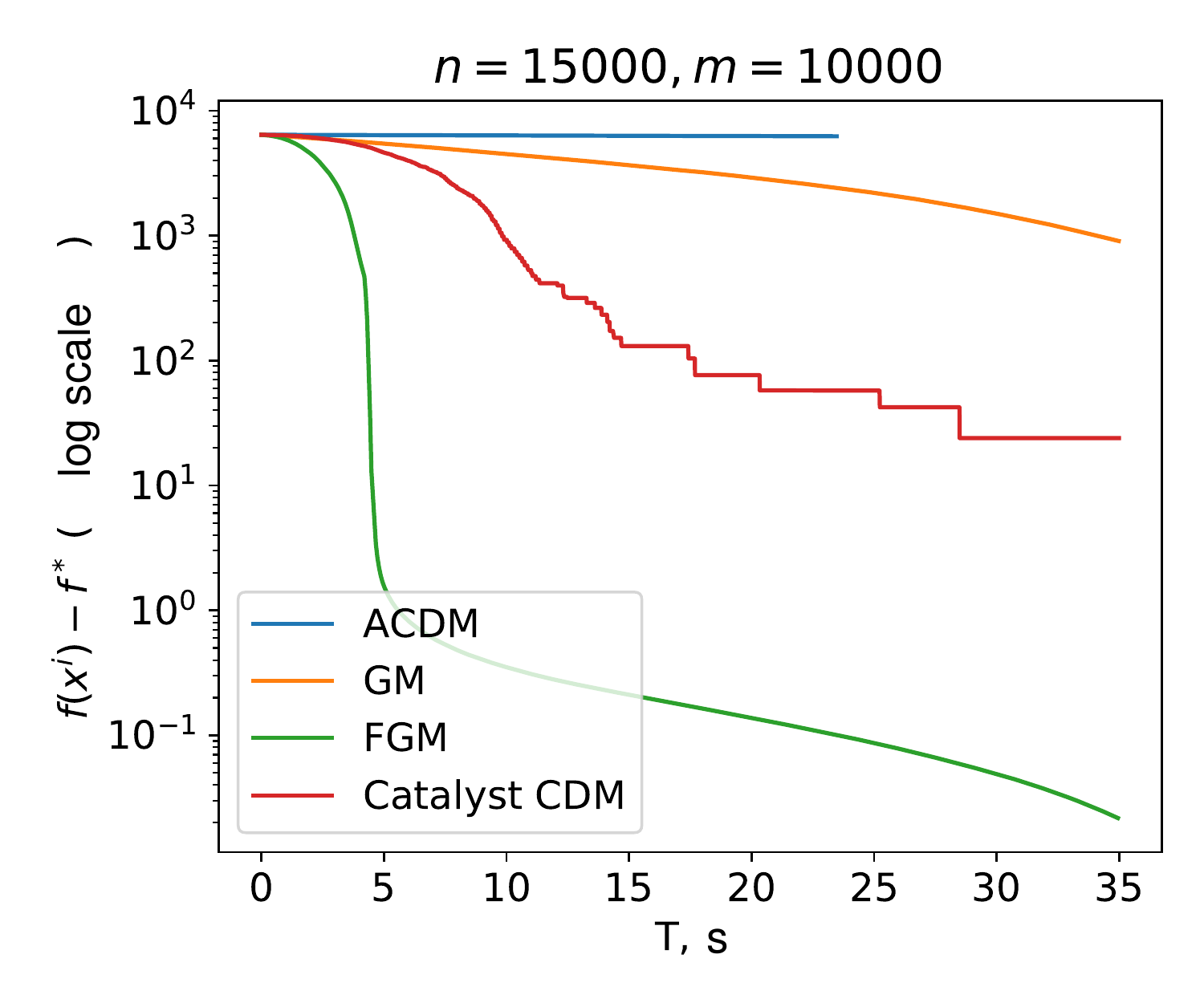}
    \caption{Convergence of methods for the SoftMax problem~\eqref{eq:softmax} with a uniformly sparse random matrix.}
    \label{fig:umc1}
\end{minipage}
\hfill
\begin{minipage}{0.49\textwidth}
    \centering
    \includegraphics[width=\linewidth]{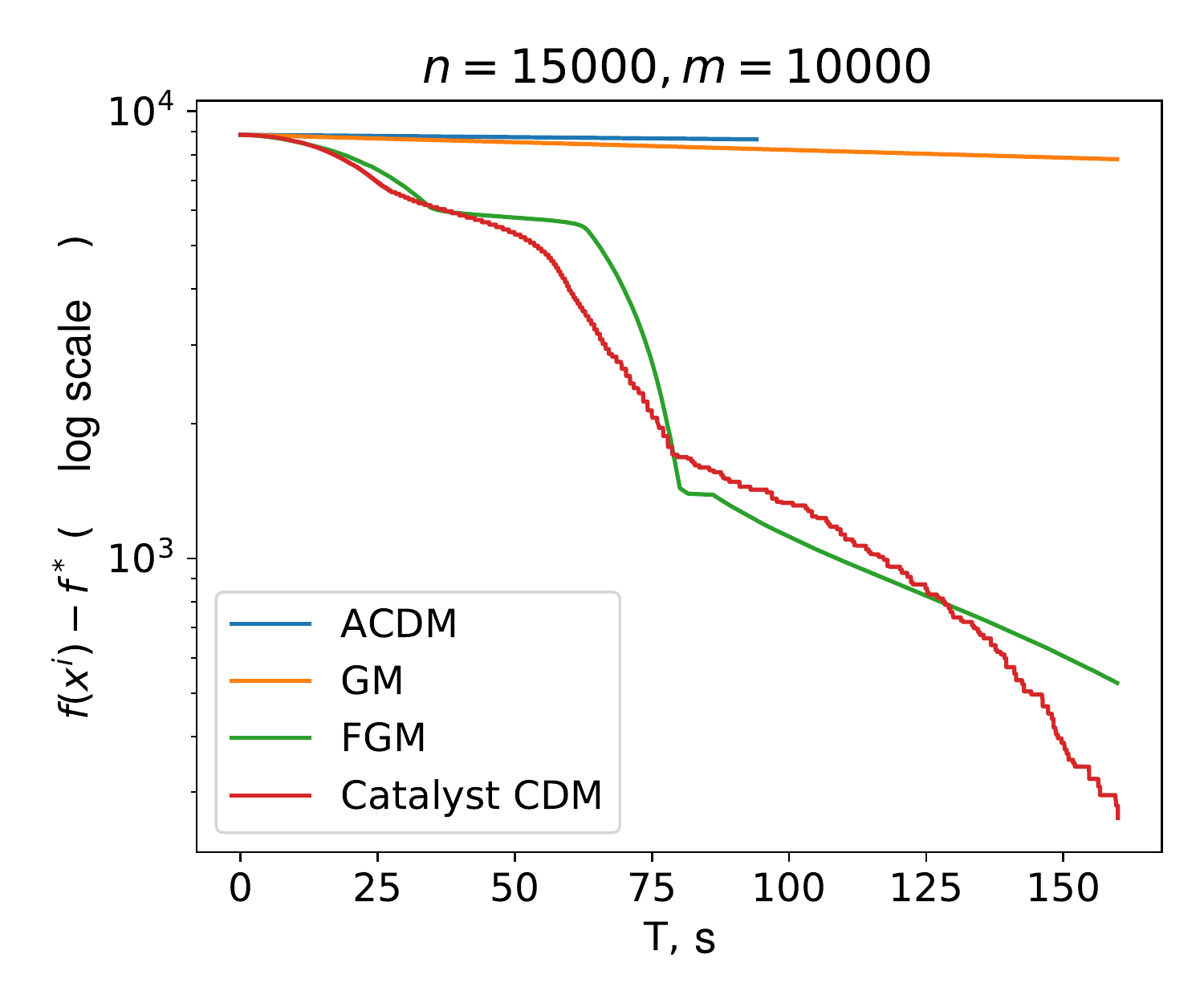}
    \caption{Convergence of methods for the SoftMax problem~\eqref{eq:softmax} with heterogeneously sparse matrix.}
    \label{fig:umc2}
\end{minipage}
\end{figure}

In Fig.~\ref{fig:umc1}, the case is presented for which all elements of the matrix $A$ are i.i.d. random variables from the discrete uniform distribution $A_{ji} \in \mathcal{U} \{0,1\}$, the number of nonzero elements is $s \approx 0.2 m$, and the parameter $\gamma = 0.6$ (as well as in the second case). In this setting, the proposed method demonstrates faster convergence compared to all methods under consideration, except the FGM. At the same time, in the setting shown in Fig.~\ref{fig:umc2}, in which the number of nonzero elements, in comparison with the first case, is increased to $ s \approx 0.75 m $, and the matrix is generated heterogeneously in accordance with the rule: $0.9 m$ rows with $0.1 n$ nonzero elements and $0.1 m$ rows with $0.9 n$ nonzero elements, and also one row of the matrix is completely nonsparse, the proposed method (Accelerated Meta-algorithm with coordinate descent as internal method) converges faster than FGM. This is explained by the fact that in this case $L = n$, but $\overline{L}$ is still quite small and, as a result, the constant in the proposed method has a noticeably smaller effect on the computational complexity than in the case of FGM. From the results of the experiment, it can also be noted that the character of its componentwise smoothness affects the efficiency of the proposed method much more significantly than the sparseness of the problem.

\subsection{Review of the results of related works}

In view of the conceptual simplicity of the scheme of the universal accelerated proximal envelope, especially in some of its special cases, such as those described in \cite{lin2015universal}, several variants of its program implementation were recently proposed, and also experiments were presented that represent the efficiency of these envelopes in the case of some typical problems. In order to capture the picture of situation in its entirety and to indicate the influence of subtle factors of implementation and adjustment of methods on their effectiveness, it makes sense to analyze some of the remarks and results of papers that consider ideologically similar problems in a slightly modified software environment. 

Let us refer to the recent review \cite{daspremont2021acceleration}, and specifically to the Catalyst clause of Section 5.6, offering remarks that are useful for the implementation of Catalyst-like methods, to the class of which the considered in this work special case of the ``Accelerated Meta-algorithm'' envelope belongs. One of the assumptions in this section is that efficient implementations of Catalyst-type methods (in particular, the implementation from \cite{mairal2019cyanure}) rely in solving auxiliary problems on the stop conditions, which are of an absolute nature, as opposed to relative (multiplicative) criteria, classically used to analyze the effectiveness of these methods. 
 
The proposed in \cite{mairal2019cyanure} approach is to calculate the duality gap, it is based on the construction described in Appendix D.2 of \cite{mairal2010sparse}. Following it, based on the information about the function at the current point, it is possible to obtain an upper bound for the absolute residual by function. The proposed solution is really elegant and possible, however, it has a number of drawbacks that do not allow us to call this solution quite convenient:
\begin{enumerate}
    \item the expression for the duality gap is computationally complex and implies the calculation of several matrix-vector
    products for potentially large dimensions;
    \item the duality gap in the described version can take on infinite values;
    \item the proposed relaxed version of the duality gap is even more computationally difficult, since
    requires additional implementation of the quadratic knapsack solver.
\end{enumerate}

At the same time, in this work there is a result that allows us to change the view on the problem of choosing the accuracy of solving the auxiliary problem (note that this problem, generally speaking, in the case of many other approaches remains open). Namely, in Corollary \ref{cor_num} it was proved that to obtain a solution to an auxiliary problem that satisfies a certain condition sufficient for its use within the outer shell, it is sufficient to perform a fixed number of iterations (depending only on the characteristics of the problem being solved). Moreover, this number (even in a certain strict sense) is small (if we discard the proportionality of the type $\sim n$, which is inevitable when using component-wise methods). Thus, following this scheme of proof, one can completely get away from the need to choose $\varepsilon_{tol}$ to the possibility of choosing the parameter $N$ is the number of iterations of the auxiliary method, without worrying about changing this parameter with an increase in the number of external iterations of the method. The practical benefit of this remark is that for many applied problems it is sometimes difficult to compute the constants characterizing the problem ($L$, $\overline{L}$), especially when optimizing functions for which the analytical expression is not known in advance. In this case, the $\varepsilon_{tol}$ parameter (as well as $ N $) becomes the hyperparameter of the algorithm, and its selection becomes the task of the engineer applying the method. However, the search space for $ N $ is discrete and essentially small, in contrast to the continually large set of  $\varepsilon_{tol}$ values variants, and both with manual selection or selection using techniques such as grid search, the problem of selecting $N$ turns out to be much simpler. 

Let us consider the results of experiments presented in the previously mentioned work \cite{mairal2019cyanure}. They are remarkable in that they assess the acceleration effect using the Catalyst envelope methods for optimization problems for sum-like functions with different regularizers, that is, statements close to machine learning problems, using some practical model examples. The general conclusion from the experiments carried out in this article is as follows: accelerating SVRG using the Catalyst envelope has no effect in the case when the optimal problem regularization parameter is $ \approx 1 / n $. It is important to note, however, that in the same situation, the MISO \cite{mairal2015incremental} method can be accelerated. In the work itself, for a different setting, an explanation is given for this fact: MISO better handle sparse matrices (no need to code lazy update strategies, which can be painful to implement). At the same time, as you can see from the figures, the Catalyst loss to directly accelerated SVRG is small. Hence, we can conclude that in the case when the implementation of lazy operations for the problem under consideration is not so onerous (the SoftMax case, considered in detail in this article, is just acceptable, besides the regularization parameter in this setting is equal to zero), the use of Catalyst remains quite justified and efficient.

\subsection{Application to Optimization of Markov Decision Processes} \label{mdp}

We denote an MDP instance by a tuple $\mathcal{MDP} := (\mathcal{S}, \mathcal{A}, \mathcal{P}, r, \gamma)$ with components defined as follows:
\begin{enumerate}
    \item $\mathcal{S}$ is a finite set of states, $|\mathcal{S}| = n$;
    \item $\mathcal{A} = \bigcup_{i \in \mathcal{S}} \mathcal{A}_i$ is a finite set of actions that is a collection of sets of actions $\mathcal{A}_i$ for states $i \in \mathcal{S}$, $|\mathcal{A}| = m$;
    \item $\mathcal{P}$ is the collection of state-to-state transition probabilities where $\mathcal{P} := \{p_{ij}(a_i) | i,j \in \mathcal{S}, a_i \in \mathcal{A}_i\}$;
    \item $r$ is the vector of state-action transitional rewards where $r \in [0, 1]^\mathcal{A}$, $r_{i,a_i}$ is the instant reward received when taking action ai at state $i \in \mathcal{S}$;
    \item $\gamma$ is the discount factor of MDP, by which one down-weights the reward in the next future step. When $\gamma \in (0, 1)$, we call the instance a discounted MDP (DMDP) and when $\gamma = 1$, we call the instance an average-reward MDP (AMDP).
\end{enumerate}

Let as denote by $P \in \mathbb{R}^{\mathcal{A} \times \mathcal{S}}$ the state-transition matrix where its $(i, a_i)$-th row corresponds to the transition probability from state $i \in \mathcal{S}$ where $a_i \in \mathcal{A}_i$ to state $j$. Correspondingly we use $\hat{I}$ as the matrix with $a_i$-th row corresponding to $e_i$, for all $i \in \mathcal{S}$, $a_i \in \mathcal{A}_i$. Our goal is to compute a random policy which determines which actions to take at each state. A random policy is a collection of probability distributions $\pi := \{\pi_i\}_{i\in\mathcal{S}}$ , where $\pi_i \in \Delta_{|\mathcal{A}_i|}$, $\pi_i(a_j)$ denotes the probability of taking $a_j \in \mathcal{A}_i$ at state $i$. One can extend $\pi_i$ to the set of $\Delta_m$ by filling in zeros on entries corresponding to other states $j \neq i$. Given an MDP instance $\mathcal{MDP} = (\mathcal{S}, \mathcal{A}, \mathcal{P}, r, \gamma)$ and an initial distribution over states $q \in \Delta_n$ , we are interested in finding the optimal $\pi^\star$ among all policies $\pi$ that maximizes the following cumulative reward $\overline{v}_\pi$ of the MDP:
\begin{align*}
    &\pi^\star := \arg \max_\pi \overline{v}^\pi,\\
    &\overline{v}^\pi := \left\{
\begin{array}{*{3}c}
   &\displaystyle \mathbb{E}^\pi \left[\sum_{t=1}^\infty \gamma^{t-1} r_{i_t, a_t} \big| i_1 \sim q \right] & \text{ in the case of } DMDP,\\
   &\displaystyle \lim_{T \rightarrow \infty} \frac{1}{T}\mathbb{E}^\pi \left[\sum_{t=1}^T r_{i_t, a_t} \big| i_1 \sim q \right] & \text{ in the case of } AMDP.
\end{array}
\right.
\end{align*}

For AMDP, $\overline{v}^\star$ is the optimal average reward if and only if there exists a vector $v^\star = (v_i^\star)_{i \in \mathcal{S}}$ satisfying its corresponding Bellman equation \cite{bertsekas1995dynamic}:
\[
    \overline{v}^\star + v_i^\star = \max_{a_i \in \mathcal{A}_i} \left\{ \sum_{j \in \mathcal{S}} p_{ij}(a_i) v_j^\star + r_{i, a_i} \right\}, \forall i \in \mathcal{S}.
\]
For DMDP, one can show that at optimal policy $\pi^\star$, each state $i \in \mathcal{S}$ can be assigned an optimal cost-to-go value $v_i^\star$ satisfying the following Bellman equation:
\[
    v_i^\star = \max_{a_i \in \mathcal{A}_i} \left\{ \sum_{j \in \mathcal{S}} \gamma p_{ij}(a_i) v_j^\star + r_{i, a_i} \right\}, \forall i \in \mathcal{S}.
\]
One can further write the above Bellman equations equivalently as the following primal linear programming problems.
\begin{align*}
    &\min_{\overline{v}, v} \quad\overline{v} &\; \text{s.t.}\;\;\; \overline{v} \textbf{1} + (\hat{I} - P) v - r \geq 0&\quad \textbf{(LP AMDP)}\\
    &\min_{v} \quad(1 - \gamma)q^\top v &\; \text{s.t.}\;\;\; (\hat{I} - \gamma P)v + r \geq 0 &\quad\textbf{(LP DMDP)}
\end{align*}
By standard linear duality, we can recast the problem formulation using the method of Lagrangian multipliers, as bi-linear saddle-point (minimax) problem. The equivalent minimax formulations are
\begin{align*}
    &\min_{v \in \mathbb{R}^n} \left\{ F(v)= \max_{\mu \in \Delta_m} \left(\mu^\top ((P - \hat{I})v + r)\right) \right\} &\quad &\textbf{(AMDP)}\\
    &\min_{v \in \mathbb{R}^n} \left\{ F_{\gamma}(v)= \max_{\mu \in \Delta_m} \left((1 - \gamma)q^\top v +  \mu^\top ((\gamma P - \hat{I})v + r)\right) \right\} &\quad &\textbf{(DMDP)}
\end{align*}

Then, one can apply the Nesterov smoothing technique to the presented max-type functional, according to \cite{nesterov2005smooth}. (The calculation is presented for the case of DMDP, as a more general one. A smoothed version of the AMDP functional is obtained similarly with the notation $A := P - \hat{I}$ and $\gamma = 1$):
\begin{align*}
    F_{\gamma}(v) &= \max_{\mu \in \Delta_m} \left(\underbrace{(1 - \gamma)q^\top}_{b} v +  \mu^\top (\underbrace{(\gamma P - \hat{I})}_{A} v + r)\right) = \max_{\mu \in \Delta_m} \left(\sum_{j=1}^m \mu_j (\left[A v\right]_j + r_j)\right) + \langle b, v\rangle\\
    &\rightarrow \max_{\mu \in \Delta_m} \left(\sum_{j=1}^m \mu_j (\left[A v\right]_j + r_j) - \sigma \sum_{j=1}^m \mu_j \ln\left(\frac{\mu_j}{1/m}\right)\right) + \langle b, v\rangle =\\
    &= \sigma \ln \left( \sum_{j=1}^m \exp\left(\frac{\left[A v\right]_j + r_j}{\sigma}\right) \right) - \sigma \ln m - \langle b,v \rangle =: f_{\gamma}(v),\text{ where }\sigma := \varepsilon / (2 \ln{m}).
\end{align*}

The resulting problem has the form of a SoftMax function, discussed in detail in the previous section. Taking into account the form of the matrix A, we can calculate the average component-wise Lipschitz constant:
\[
    P_{ji} \in [0, 1], \hat{I}_{ji} \in \{0, 1\}\quad\Longrightarrow\quad\overline{L} = \frac{1}{\sigma} \frac{1}{n} \sum_{i=1}^n \max_{j=1,...,m} |A_{ji}| \leq \frac{\gamma}{\sigma} = \frac{2 \gamma \ln m}{\varepsilon}
\]

So, one can get the following estimates, which are also given in the Tables~\ref{table_compare2},~\ref{table_compare2b} (transition from the duality gap accuracy $\varepsilon$ in matrix game setting to the $\widetilde{\varepsilon}$ accuracy to obtain the $\widetilde{\varepsilon}$-approximate optimal policy satisfying the condition in expectation $\mathbb{E}\overline{v}^\pi \geq \overline{v}^\star - \widetilde{\varepsilon}$ is carried out according to the rules described in more detail in the work \cite{jin2020efficiently}):
\begin{align*}
    \varepsilon \sim \frac{1}{2} \cdot \frac{\widetilde{\varepsilon}}{3}\quad\Longrightarrow\quad&T_{CCDM} \;\;= \widetilde{\mathcal{O}}\left(\text{nnz}(P) \sqrt{\log m} \cdot \widetilde{\varepsilon}^{-1} \right),\\
    \varepsilon_{\gamma} \sim \frac{1}{2} \cdot \frac{(1 - \gamma)\widetilde{\varepsilon}}{3}\quad\Longrightarrow\quad&T_{CCDM, \gamma} = \widetilde{\mathcal{O}}\left(\gamma^{1/2} (1-\gamma)^{-1} \cdot \text{nnz}(P) \sqrt{\log m} \cdot \widetilde{\varepsilon}^{-1} \right),
\end{align*}
where $\text{nnz}(P) \leq n \cdot m$ denotes the number of nonzero elements in matrix $P$. In addition to the result corresponding to the considered method, the Tables~\ref{table_compare2},~\ref{table_compare2b} contain complexity bounds of other known approaches to solving matrix games and MDP problems (for the sake of compactness, it is used the notation $\text{nnz}'(P) = \text{nnz}(P) + (m + n) \log^3 (mn)$). It can be seen that for the case of $\gamma = 1$, the described approach allowing to obtain one of the best among the known estimates, and in the case of $\gamma < 1$ it is close in efficiency to many modern approaches. Moreover, to describe the method used in this article, it was enough to apply only a special case of the universal accelerated proximal envelope for the classical coordinate descent method. This approach is conceptually much simpler than the other methods cited here (which, by the way, are often applicable only to very particular settings), and allows one to obtain complexity bounds for AMDP problem that notedly competitive with the best alternatives. 

\begin{table}[ht]
\centering

\caption{Comparison of the effectiveness of approaches ($\gamma=1$ case)}
\label{table_compare2}
\begin{tabular}{|c|c|}
\hline
Computational complexity & Source \\ \hline

$\displaystyle \widetilde{\mathcal{O}}\left(\text{nnz}(P) \sqrt{\log m} \cdot \widetilde{\varepsilon}^{-1} \right)$ & this paper \\ \hline

$\displaystyle \widetilde{\mathcal{O}}\left(\text{nnz}(P) \sqrt{m / n} \cdot \widetilde{\varepsilon}^{-1} \right)$ & \cite{carmon2019variance}  \\ \hline

\;\;$\displaystyle \widetilde{\mathcal{O}}\left(\log^3 (mn)\;\sqrt{\text{nnz}(P) \cdot \text{nnz}'(P)} \cdot \widetilde{\varepsilon}^{-1} \right)$\;\; & \cite{carmon2020coordinate} \\ \hline
\end{tabular}
\end{table}
\vspace{-1.0cm}

\begin{table}[ht]
\centering
\caption{Comparison of the effectiveness of approaches ($\gamma \in (0, 1)$ case)}
\label{table_compare2b}
\begin{tabular}{|c|c|}
\hline
Computational complexity & Source \\ \hline

$\displaystyle \widetilde{\mathcal{O}}\left(\gamma^{1/2} (1-\gamma)^{-1} \; \text{nnz}(P) \sqrt{\log m} \cdot \widetilde{\varepsilon}^{-1} \right)$ & this paper \\ \hline

$\displaystyle \widetilde{\mathcal{O}}\left(\gamma (1-\gamma)^{-1}\;\text{nnz}(P) \sqrt{m / n} \cdot \widetilde{\varepsilon}^{-1} \right)$ & \cite{carmon2019variance}  \\ \hline

\;\;$\displaystyle \widetilde{\mathcal{O}}\left(\gamma (1-\gamma)^{-1} \;\log^3 (mn)\;\sqrt{\text{nnz}(P) \cdot \text{nnz}'(P)} \cdot \widetilde{\varepsilon}^{-1} \right)$ & \cite{carmon2020coordinate}\;\; \\ \hline

$\displaystyle \widetilde{\mathcal{O}}\left(nm\;\left(n + \left(1-\gamma\right)^{-3}\right) \cdot \log{\left(\widetilde{\varepsilon}^{-1}\right)} \right)$ & \cite{sidford2018variance}  \\ \hline
\end{tabular}
\end{table}
\vspace{-1.0cm}

\section{Conclusion} 

In this paper, we propose a version of the Coordinate Descent Method, accelerated using the universal proximal envelope ``Accelerated Meta-algorithm''. The performed theoretical analysis of the proposed method allows us to assert that the dependence of its computational complexity on the dimension of the problem and the required solution accuracy is not inferior to other methods used to optimize convex Lipschitz smooth functions, and the computational complexity is comparable to that of the Fast Gradient Method. At the same time, the proposed scheme retains the properties of the classical Coordinate Descent Method, including the possibility of using the properties of componentwise smoothness of the function. The given numerical experiments confirm the practical efficiency of the method, and also emphasize the particular relevance of the proposed approach for the problem of optimizing a SoftMax-like function that often arises in various applications. As an example of such an application, it was considered the problem of optimizing the MDP, and using the described approach, a method was proposed for solving the averaged version of the MDP problem, which gives a complexity bound that competes with the most efficient known ones. 

\bibliographystyle{splncs04}
\bibliography{main}

\appendix
\addcontentsline{toc}{section}{Appendices}
\section*{Appendices}
\section{Proof of the Corollary \ref{cor_num}} \label{cor_num_proof}

Due to the $(H+L)$-Lipschitz smoothness of the  function $F$ it holds the following inequality: $\displaystyle F(y_0) - F(y_*) \leq \frac{H+L}{2} \|y_0 - y_*\|_2^2$. Using this inequality together with the estimation \eqref{f_conv}, one can write out a condition of achieving the required accuracy $\widetilde{\varepsilon}$ with respect to the function value: $\displaystyle \frac{H+L}{2} \left(1 - \frac{1}{\kappa}\right)^{N} \|y_0 - y_*\|_2^2 \leq \widetilde{\varepsilon}$. Also, the relation $1 - 1/\kappa \leq \exp\{-1/\kappa\}$ holds and therefore one can state out: 
    $$\displaystyle \frac{H+L}{2} \exp\{\kappa / N\} \|y_0 - y_*\|_2^2 \leq \widetilde{\varepsilon}.$$ 
    By taking the logarithm and substituting the expression for $\kappa$, we get an expression for the number of iterations from $\widetilde{\varepsilon}$:
    \begin{equation} \label{itt}
        N(\widetilde{\varepsilon}) = \ceil[\Bigg]{\frac{Z}{H} \ln {\left\{\frac{(H+L) \|y_0 - y_*\|_2^2}{2 \widetilde{\varepsilon}}\right\}}}.
    \end{equation}
    Due to the $H$-strongly convexity of the function $F$ it holds the following inequality $\displaystyle \|\overline{y}_N - y_*\|_2^2 \leq \frac{2}{H} (F(\overline{y}_N) - F(y_*))$, where $\overline{y}_N = \mathbb{E}[y_N]$. The function $F$ is convex, and therefore $F(\overline{y}_N) \leq \mathbb{E}[F(y_N)]$ (Jensen inequality), and from this, together with \eqref{crit}, one can obtain a sufficient condition of achieving a solution to the auxiliary problem:
    $$\mathbb{E}[F(y_N)] - F(y_*) \leq \frac{H}{2} \left(\frac{H}{3H + 2L}\right)^2 \|y_0 - y_*\|_2^2.$$
    Substituting into the formula \eqref{itt} instead of $\widetilde{\varepsilon}$ the expression from the right-hand side of this inequality, we arrive at the desired result.

\section{Proof of the Theorem \ref{th4}} \label{th4_proof}

Let us denote by $N$ the number of iterations of the internal method, performed in a particular random case until the stop condition is satisfied. Following the context of the proof of Corollary \ref{cor_num} and applying Markov inequality, we obtain the following estimate for the probability that the number of iterations of interest to us will exceed some given $k$:  

    \begin{align*}
        Pr(N \geq k) &\leq Pr\left(F(y_k) - F(y_*) \geq \frac{H}{2} \left(\frac{H}{3H + 2L}\right)^2 \|y_0 - y_*\|_2^2\right)\\
        &\leq \min\left\{1,\;\;\frac{\mathbb{E}[F(y_k)] - F(y_*)}{\frac{H}{2} \left(\frac{H}{3H + 2L}\right)^2 \|y_0 - y_*\|_2^2}\right\}\quad\text{(Markov inequality)}\\
        &\leq \min\left\{1,\;\;\left(1 + \frac{L}{H}\right) \left(3 + \frac{2L}{H}\right)^2 \exp\left\{\frac{\kappa}{k}\right\}\right\}
    \end{align*}
    Further, we can estimate by the infinite series the expectation of the number of iterations of the internal method:
    
    \begin{align*}
        \mathbb{E} [\hat{N}] &= \sum_{k=1}^\infty Pr(N \geq k) \\
        &\leq \int_0^{N(\widetilde{\varepsilon})} dk + \left(1 + \frac{L}{H}\right) \left(3 + \frac{2L}{H}\right)^2 \exp\left\{\kappa\right\} \cdot \int_{N(\widetilde{\varepsilon})}^\infty \exp\left\{-k\right\}\\
        &= N(\widetilde{\varepsilon}) + \left(1 + \frac{L}{H}\right) \left(3 + \frac{2L}{H}\right)^2 \exp\left\{\frac{\kappa}{N(\widetilde{\varepsilon})}\right\} \leq N(\widetilde{\varepsilon}) + 1.
    \end{align*}
    The number of external iterations of the method is constant, which means that when calculating the mean of the total number of iterations, this value can be taken out of the mean by applying for each internal iteration a single estimate given above, whence the proved statement follows.
    
\section{Proof of the Theorem \ref{th5}} \label{th5_proof}

Let us formulate the following obvious auxiliary proposition.

\begin{proposition} \label{constr}
    The algorithmic complexity of the accelerated meta-algorithm with coordinate descent as an internal method is
    $$T = \mathcal{O}\left(\widetilde{N} (T_{out} + N_{\delta} T_{inn})\right),$$
    where $T_{out}$ is the amortized algorithmic complexity of computations performed at the iteration of the accelerated meta-algorithm, $T_{inn}$ is the amortized algorithmic complexity of the iteration of the coordinate descent algorithm.
\end{proposition}
Let us now reformulate the estimate from the Proposition~\ref{constr} as follows:
    $$T = \mathcal{O}\left(\hat{N} \cdot T_{iter}\right),$$
    where $T_{iter}$ is the amortized algorithmic complexity of an elementary iteration of the accelerated meta-algorithm, that is, an iteration that can be both an internal iteration of the coordinate descent algorithm and the main iteration of the meta-algorithm. The complexity of the main iteration of the meta-algorithm is determined, first of all, by calculating the full gradient vector of $f$, and the complexity of this procedure (from the condition of the theorem) is $\mathcal{O}(s \cdot n)$. At the same time, due to the $Z = n \overline{L}$, also holds $N_\delta = \widetilde{\mathcal{O}}(n)$, where symbol $\widetilde{\mathcal{O}}(\cdot)$ means the same as $\mathcal{O}(\cdot)$, but with the possible presence of factors of $\log(\cdot)$ order. Since the main iteration of the meta-algorithm is performed every $N_\delta$ elementary iterations, where $N_\delta$ is constant, this, using any of the methods of amortization analysis, trivially yields an amortized algorithmic complexity of the main iteration of the meta-algorithm, which is $\mathcal{O}(s)$. The complexity of the coordinate descent algorithm's iteration (if all operations performing ``in place``, instead of copying the values of every point, it is quite acceptable for this construction) is determined by calculating one component of the gradient, and is also $\mathcal{O}(s)$. From this obtain $T_{iter} = \mathcal{O}(s)$. Using the estimate \eqref{total_iters} and substituting the optimal value of $H$, we obtain the provided complexity of the method.

\end{document}